\documentclass[12pt]{amsart}

\usepackage{graphicx,amsmath,amscd,amsfonts,amsthm,amssymb,verbatim,stmaryrd,fullpage}
\newtheorem{theorem}{Theorem}

\newtheorem{prop}[theorem]{Proposition}

\theoremstyle{remark}

\newcommand{\R}{\mathbb R}
\newcommand{\C}{\mathbb C}

\newcommand{\Z}{\mathbb Z}
\newcommand{\Q}{\mathbb Q}

\newcommand{\A}{\mathbb A}
\newcommand{\bH}{\mathbb H}

\newcommand{\p}{\mathfrak p}
\newcommand{\g}{\mathfrak g}

\newcommand{\ga}{\mathfrak a}

\newcommand{\gn}{\mathfrak n}

\newcommand{\gf}{\mathfrak f}
\newcommand{\gm}{\mathfrak m}

\newcommand{\cO}{\mathcal O}

\newcommand{\tr}{\text{tr}}

\newcommand{\ord}{\text{ord}}

\newcommand{\Vol}{\text{Vol}}

\newcommand{\be}{\begin{equation}}
\newcommand{\ee}{\end{equation}}
\newcommand{\bes}{\begin{equation*}}
\newcommand{\ees}{\end{equation*}}

\newcommand{\ba}{\begin{align}}
\newcommand{\ea}{\end{align}}
\newcommand{\bas}{\begin{align*}}
\newcommand{\eas}{\end{align*}}

\title{Endoscopy and cohomology growth on $U(3)$}
\author{Simon Marshall}

\begin{document}

\begin{abstract}
We apply the endoscopic classification of automorphic forms on $U(3)$ to prove a case of a conjecture of Sarnak and Xue on cohomology growth.
\end{abstract}

\maketitle

\section{Introduction}

Let $U(p,q;\R)$ denote the real unitary group of signature $(p,q)$, and define $\bH_\C$ to be the globally symmetric space $U(2,1;\R)/( U(2;\R) \times U(1;\R))$.  Let $\Gamma \subset U(2,1; \R)$ be an arithmetic congruence lattice arising from a Hermitian form in three variables with respect to a CM extension $E/F$.  If $\cO$ is the ring of integers of $F$ and $\gn \subseteq \cO$ is an ideal, we may define principal congruence subgroups $\Gamma(\gn) \subseteq \Gamma$, and let $Y(\gn)$ be the arithmetic locally symmetric space $\Gamma(\gn) \backslash \bH_\C$.  We give the precise definition of these objects, and the statement of Theorem \ref{main} below, in Section \ref{notnquotient}.  Let $V(\gn) = | \Gamma : \Gamma(\gn)|$, which is asymptotically equal to the volume of $Y(\gn)$.  We let $H^1_{(2)}(Y(\gn), \C)$ be the space of square integrable harmonic 1-forms on $Y(\gn)$, and let $\beta^1_{(2)}(Y(\gn))$ be its dimension.

When $\Gamma$ is cocompact, Sarnak and Xue \cite{SX} have made a general conjecture on the asymptotic multiplicities of automorphic forms which implies that $\beta^1_{(2)}(Y(\gn)) \ll_\epsilon V(\gn)^{1/2 + \epsilon}$, and in the case of $U(2,1; \R)$ they are able to prove the weaker bound $\beta^1_{(2)}(Y(\gn)) \ll_\epsilon V(\gn)^{7/12 + \epsilon}$.  This paper settles their conjecture in this case, by proving the following upper bound on $\beta^1_{(2)}(Y(\gn))$.

\begin{theorem}
\label{main}

We have $\beta^1_{(2)}(Y(\gn)) \ll V(\gn)^{3/8}$, and there exists $\Gamma$ such that $\beta^1_{(2)}(Y(\gn)) \gg V(\gn)^{3/8}$.

\end{theorem}

The proof of Theorem \ref{main} relies on the endoscopic classification of automorphic representations on $U(3)$ by Rogawski\footnote{We note that Flicker claims to have found an error in Rogawski's work (\cite{F}, p. 393), which he also claims to fix.  We shall simply take the results of the endoscopic classification as proven, and leave the question of attribution to those with more expertise than us.},
combined with a slight generalisation of the fundamental lemma for the pair $(U(3), U(2) \times U(1))$.  The essential idea is that the automorphic forms contributing to $H^1_{(2)}(Y(\gn), \C)$ in Matsushima's formula are nontempered, and Rogawski shows that they are all transfers of one dimensional representations on the endoscopic group $U(2) \times U(1)$ of $U(3)$.  Our work lies in making this result quantitative.  Note that Rogawski also proves that $\beta^1_{(2)}(Y(\gn)) = 0$ if $Y(\gn)$ arises from a nine-dimensional division algebra with involution over $E$, and when combined with Theorem \ref{main} this provides an understanding of the growth of $\beta^1_{(2)}$ for all arithmetic congruence lattices in $U(2,1; \R)$.

We remark that if $\gn = \ga \p^k$ with $\ga$ and $\p$ fixed, $\p$ prime, and $k$ growing, Theorem \ref{main} could probably be derived by combining the results of Gelbart, Rogawski, and Soudry \cite{GRS}, which describe endoscopic $L$-packets in terms of the theta lift from $U(2)$ to $U(3)$, with the main theorem of \cite{CM}.

{\bf Acknowledgements}: We would like to thank Frank Calegari, Dihua Jiang, and Matthew Stover for helpful discussions, and Peter Sarnak for the suggestion to study cohomology growth using Rogawski's work.

\section{Notation}

\subsection{Number fields}

Let $E/F$ be a CM extension of number fields, with $\cO_E$ and $\cO = \cO_F$ their rings of integers and $\A_E$ and $\A = \A_F$ their rings of adeles.  We denote the maximal compact subrings of the finite adeles $\A_{E,f}$ and $\A_f$ by $\widehat{\cO}_E$ and $\widehat{\cO}$.  Let $N$ be the norm map from $E$ to $F$, $\A_E^1$ the group of norm 1 ideles of $E$, and $I_E^1 = \A_E^1 / E^1$.  We shall denote places of $E$ and $F$ by $w$ and $v$ respectively, with corresponding completions $E_w$ and $F_v$, and define $E_v = E \otimes_F F_v$.

Fix a character $\mu$ of $\A_E / E^\times$ whose restriction to $\A/F^\times$ is the character associated to $E/F$ by class field theory.  Let $S_f$ be a set of finite places of $F$ containing all places at which $E/F$ is ramified, all places below those at which $\mu$ is ramified, all $v | 2$, and at least one place that is nonsplit in $E$.  Let $S_\infty$ be the set of infinite places of $F$, and let $S = S_\infty \cup S_f$.

\subsection{Unitary groups}

Let $\Phi_n = (\Phi_{ij})$, where $\Phi_{ij} = (-1)^{i-1} \delta_{i, n+1-j}$ and $\delta_{a,b}$ is the Kronecker delta.  If $x \in E$ satisfies $\tr_{E/F}(x) = 0$, then $\Phi_n$ is a Hermitian form with respect to $E/F$ if $n$ is odd, and $x \Phi_n$ is Hermitian if $n$ is even.  We let $U(n)$ be the unitary group of this Hermitian form, so that

\bes
U(n)(F) = \{ g \in GL_n(E) | g \Phi_n {}^{t}\overline{g} = \Phi_n \}.
\ees
For any ideal $\gn \subseteq \cO$, we define the compact subgroup $U(n,\gn) \subset U(n)(\A_f)$ by

\bes
U(n,\gn) = \{ g \in U(n)(\widehat{\cO}) \subset GL_n(\widehat{\cO}_E) | g \equiv I_n (\gn \widehat{\cO}_E) \}.
\ees
We shall denote $U(3)$ by $G^*$.  If $\gn \subset \cO$ is an ideal, define the compact subgroup $K^*(\gn) = \otimes_v K^*_v(\gn)$ of $G^*(\A)$ by setting $K^*_v(\gn) = U(2;\R) \times U(1;\R)$ if $v | \infty$, and $\otimes_{v \nmid \infty} K^*_v(\gn) = U(3,\gn)$.

Choose a place $v_0 \in S_\infty$.  Let $\Phi$ be a Hermitian form on $E^3$ with respect to $E/F$ that is indefinite at $v_0$ and definite at all other real places of $F$, and let $G$ be the unitary group of $\Phi$.  It is known that the isomorphism class of $G$ over $F$ depends only on the extension $E/F$ and the place $v_0$, see for instance Section 1.2 of \cite{PY}; in particular, $G$ is quasi-split iff $F = \Q$.

If $v$ is a finite place of $F$ that splits in $E$, then there are isomorphisms from $G_v$ and $G^*_v$ to $GL_3(F_v)$ that are canonical up to inner automorphism.  If $v$ is finite and nonsplit in $E/F$, it follows from a theorem of Landherr \cite{L} that there is a unique Hermitian form on $E_v^3$ with respect to $E_v/F_v$, and this gives an isomorphism from $G_v$ to $G^*_v$ that is canonical up to inner automorphism.  If we let $K = \otimes_v K_v$ be a compact open subgroup of $G(\A)$ such that $K_{v_0} = U(2;\R) \times U(1;\R)$, $K_v = U(3;\R)$ when $v_0 \neq v | \infty$, and $K_v$ is hyperspecial whenever $v \notin S$, we may therefore fix isomorphisms $\phi_v : G_v \simeq G_v^*$ for all finite $v$ such that $\phi_v K_v = K^*_v$ for $v \notin S$.

\subsection{Adelic quotients}
\label{notnquotient}

If $\gn \subset \cO$ is relatively primes to $S_f$, we define $K(\gn) = \otimes_v K_v(\gn)$ by setting $K_v(\gn) = K_v$ for $v \in S$, and $K_v(\gn) = \phi_v( K^*_v(\gn) )$ for $v \notin S$.  We define $Y(\gn)$ to be the adelic quotient $G(F) \backslash G(\A) / K(\gn) Z(\A)$.  It is a finite union of finite volume quotients of the globally symmetric space $\bH_\C$, and it is compact iff $F \neq \Q$.  If we fix a translation-invariant volume form on $\bH_\C$ and let $\Vol(Y(\gn))$ be the volume of $Y(\gn)$ with respect to this form then we have $\Vol(Y(\gn)) = c(\gn) V(\gn)$, where

\be
\label{V}
V(\gn) = |U(3,\cO) Z(\A_f) : U(3,\gn) Z(\A_f) |
\ee
and $| \log c(\gn) |$ is bounded in terms of our choice of $K_v$ for $v \in S_f$.  Note that the formulas for the orders of $GL_3$ and $U(3)$ over a finite field (see \cite{A}) imply that $N\gn^8 \ll V(\gn) \ll N\gn^8$.

With this notation, the precise statement of Theorem \ref{main} is that $\beta_{(2)}^1(Y(\gn)) \ll V(\gn)^{3/8}$, and that $\beta_{(2)}^1(Y(\gn)) \gg V(\gn)^{3/8}$ if $K_v$ are chosen small enough for $v \in S_f$.

\subsection{Endoscopic groups}

Let $H \simeq U(2) \times U(1)$ be the unique elliptic endoscopic group of $G^*$, which we consider to be embedded in $G^*$ as

\bes
\left( \begin{array}{ccc} * & & * \\ & * & \\ * & & * \end{array} \right).
\ees
We let $\det_0 : H \rightarrow U(1)$ and $\lambda : H \rightarrow U(1)$ be the maps given by the determinant on the $U(2)$ factor and projection onto the $U(1)$ factor.  We fix an embedding of $L$-groups $^{L}H \rightarrow{} ^{L}G^*$ associated to the character $\mu$ as in \cite{R1}, Section 4.8.1.  The centers of $G$ and $G^*$ will both be denoted by $Z \simeq U(1)$, and we consider $Z$ to be a subgroup of $H$ in the natural way.  As $Z(F) \backslash Z(\A) \simeq I_E^1$, $\mu$ gives a character of $Z(F) \backslash Z(\A)$ by restriction, which will also be denoted $\mu$.  We shall denote the restriction of $\mu$ to $E_w$ by $\mu_w$, and its restriction to $Z_v$ by $\mu_v$.

\subsection{Measures and function spaces}

Choose Haar measures $dg = \otimes dg_v$, $dg^* = \otimes dg_v^*$, and $dh = \otimes dh_v$ on $G(\A)$, $G^*(\A)$ and $H(\A)$ respectively, where $dg_v$ and $dg_v^*$ are equal under the isomorphism $G_v \simeq G_v^*$ at all finite places.  We assume that the local measures give mass 1 to a hyperspecial maximal compact for all $v \notin S$.  Let $dz = \otimes_v dz_v$ be the Haar measure on $Z(\A)$ that gives the maximal compact mass 1 everywhere, and let $d\overline{g} = \otimes_v d\overline{g}_v$ be the measure on $G(\A) / Z(\A)$ given by $d\overline{g}_v = dg_v / dz_v$.

For any place $v$ and a character $\omega$ of $E_v^1 \simeq Z_v$, we define $C(G_v, \omega)$ to be the space of smooth complex-valued functions $f$ on $G_v$ such that $f$ is compactly supported modulo $Z_v$, $f(zg) = \omega(z)^{-1} f(g)$, and if $v$ is infinite then $f$ is $K_v$-finite.  If $\omega$ is a character of $I_E^1$, we define $C(G, \omega)$ to be the analogous space in the global case.  The spaces $C(G^*, \omega)$ and $C(H, \omega)$ are defined similarly.

If $\pi$ is an admissible representation of $G_v$ with central character $\omega$, and $f \in C(G_v, \omega)$, we define $\pi(f)$ to be

\bes
\pi(f) = \int_{G_v / Z_v} f(g) \pi(g) d\overline{g}.
\ees

\subsection{Automorphic forms}

If $\omega$ is a unitary character of $Z(F) \backslash Z(\A) \simeq I_E^1$, we let $L^2(G,\omega)$ be the space of square integrable complex functions $\phi$ on $G(F) \backslash G(\A)$ that satisfy $\phi(zg) = \omega(z) \phi(g)$.  We let $L^2_d(G,\omega)$ be the subspace that decomposes discretely under the action of $G(\A)$.  We define $L^2_d(H, \omega)$ similarly, recording only the action of the subgroup $Z$ of $Z(H)$.  We denote the set of discrete $L$-packets on $G$ and $H$ by $\Pi(G)$ and $\Pi(H)$; see \cite{R1}, Section 12 and 13.3, for the definition and description of these sets.

\section{The packets $\Pi(\xi)$}
\label{packets}

In \cite{R1}, Sections 13 and 14, Rogawski defines an $L$-packet $\Pi(\xi) \in \Pi(G)$ for every one dimensional representation $\xi \in L^2_d(H,\omega)$ satisfying certain conditions.  In this section we recall the definition and important properties of these packets.

\subsection{Split finite places}

Let $v$ be a finite place that splits in $E/F$, so that $E_v = E_w \oplus E_{w'}$.  We identify $E_w$ with $E_{w'}$.  We have

\bes
G_v = \{ (g,h) | g, h \in GL_3(E_w), h = \Phi ^{t}g^{-1} \Phi^{-1} \},
\ees
and

\bes
Z_v = \{ ( xI, x^{-1}I) | x \in E_w^* \} \simeq E_v^1 \simeq E_w^*.
\ees
Note that under the identification $Z_v \simeq E_w^*$, we have $\mu_v(x) = \mu_w(x)^2$.

Let $\xi$ be a unitary character of $H_v \simeq GL_2(E_w) \times GL_1(E_w)$, and let $\omega$ denote its restriction to $Z_v$.  If $P$ is the standard parabolic subgroup of $G_v$ with Levi $H_v$, the local packet $\Pi_v(\xi)$ is the unitarily induced representation $I(\xi \otimes \det_0 \circ \mu_w)$ from $P$ to $G_v$ (\cite{F}, Proposition 4, p. 279).  It has central character $\omega \otimes \mu_v$, and we shall denote it by $\pi^n(\xi)$.

\subsection{Nonsplit finite places}

If $v$ is a finite place that does not split in $E/F$ and $\xi$ is a unitary character of $H_v$, the local packet $\Pi_v(\xi)$ contains two representations $\pi^n(\xi)$ and $\pi^s(\xi)$.  The representation $\pi^n(\xi)$ is nontempered, and spherical whenever all data are unramified, while $\pi^s(\xi)$ is supercuspidal.  If the restriction of $\xi$ to $Z_v$ is $\omega$, both representations in $\Pi_v(\xi)$ have central character $\omega \otimes \mu_v$.

\subsection{Real places}

We take the following results from \cite{R1}, Section 12.3.  For any real place $v$, let $t_v \in \Z$ be such that $\mu_v(z) = (z/\overline{z})^{t_v + 1/2}$.

To describe $\Pi(\xi)$ at the place $v_0$, we recall the classification of cohomological representations of $U(2,1; \R)$ (\cite{R1}, Proposition 15.2.1 and \cite{BW}, Theorem 4.11).  If $\pi$ is an irreducible unitary $G_{v_0}$-module, we have $H^1(\g, K; \pi) = 0$ unless $\pi \in \{ J^+, J^- \}$, where $J^+$ and $J^-$ are nontempered.  When $\pi = J^\pm$, we have $H^1(\g, K; \pi) = \C$ with Hodge types $(1,0)$ and $(0,1)$ respectively.  In addition, $H^2(\g, K; \pi) = 0$ unless $\pi \in \{ 1, D, D^+, D^- \}$, where $1$ is the trivial representation, and $D$, $D^+$, and $D^-$ are discrete series representations with Hodge types $(1,1)$, $(2,0)$ and $(0,2)$ respectively.

For any one-dimensional representation $\xi$ of $H_{v_0}$, the local packet $\Pi_{v_0}(\xi)$ is disjoint from $\{ J^\pm \}$ unless $\xi = (\det_0)^{-t_{v_0}-1} \lambda^1$ (case 1) or $\xi = (\det_0)^{-t_{v_0}} \lambda^{-1}$ (case 2).  In the remaining two cases, we have

\bes
\Pi_{v_0}(\xi) = \Big\{ \begin{array}{ll} \{ J^+, D^- \} & \text{in case 1,} \\ \{ J^-, D^+ \} & \text{in case 2.} \end{array}
\ees
We will denote the nontempered member of $\Pi(\xi)$ by $\pi^n(\xi)$, and the tempered member by $\pi^s(\xi)$.

At the remaining places, we have $G_v = U(3; \R)$.  The packet $\Pi_v(\xi)$ is only defined for $\xi$ of the form $(\det_0)^{p-t_v} \lambda^q$ with $p-q \ge 1$ or $q-p \ge 2$, and when it is, it consists of one irreducible representation of $G_v$ which we denote $\pi^s(\xi)$.  $\Pi_v(\xi)$ is the trivial representation exactly when $\xi$ is either $(\det_0)^{-t_v-1} \lambda^1$ or $(\det_0)^{-t_v} \lambda^{-1}$.

\subsection{Global packets}
\label{packetsglobal}

Let $\xi \in L^2_d(H, \omega)$ be a one dimensional representation.  Define the global $L$-packet $\Pi(\xi)$ to be $\otimes_v \Pi_v(\xi_v)$.  It is proven that $\Pi(\xi) \in \Pi(G)$ (\cite{R1}, Theorem 13.3.2 and Section 14), and that any representation $\pi = \otimes_v \pi_v \in L^2_d(G, \omega)$ satisfying $\pi_{v_0} \simeq J^\pm$ must lie in a packet $\Pi(\xi)$ for some $\xi$ (\cite{R1}, Theorem 13.3.6).  If $\pi = \otimes_v \pi_v \in \Pi(\xi)$, define $n(\pi)$ to be the number of places at which $\pi_v = \pi^s(\xi_v)$.  It is known (see \cite{R2} and \cite{F}, p. 218) that there is a global factor $\varepsilon(\xi, \mu) = \pm 1$ such that

\bes
m(\pi) = \tfrac{1}{2} (1 + \varepsilon(\xi, \mu) (-1)^{n(\pi)} ).
\ees
In particular, $m(\pi)$ is either 0 or 1.

\subsection{Transfers and character identities}

Suppose that $v$ is finite and $f \in C(G_v, \omega)$.  There exists a function $f^H \in C(H, \omega \mu_v^{-1})$, called a transfer of $f$, such that the unstable orbital integrals of $f$ match the stable integrals of $f^H$; see \cite{R1}, Section 4.9 for details.  Note that we define this transfer in the non-quasi-split case by applying our identification $\phi_v : G_v \simeq G^*_v$ and applying the usual transfer for $G^*$.  When $\xi$ is a character of $H_v$ such that the restriction of $\xi$ to $Z_v$ is $\omega \mu_v^{-1}$ and $v$ is split, we have  (\cite{R1}, Lemma 4.13.1)

\bes
\tr( \pi^n(\xi)(f) ) = \tr( \xi(f^H)),
\ees
and when $v$ is nonsplit we have (\cite{R1}, Corollary 12.7.4 and \cite{F}, p. 215)

\be
\label{charinert}
\tr( \pi^n(\xi)(f) ) + \tr( \pi^s(\xi)(f) ) = \tr( \xi(f^H)).
\ee

\section{Proof of Theorem \ref{main}}
\label{proof}

\subsection{The upper bound}
\label{proofupper}

We modify our notation sightly, and now define $J^\pm$ to be the representation of $G_\infty = \otimes_{v | \infty} G_v$ that is equal to $J^\pm$ at $G_{v_0}$ and trivial at all other places.  We also define $\Xi_\infty$ to be the set of characters of $H_\infty$ that are equal to either $(\det_0)^{-t_v-1} \lambda^1$ or $(\det_0)^{-t_v} \lambda^{-1}$ at each place $v$.  By Matsushima's formula, we have

\bes
\beta^1_{(2)}( Y(\gn)) = \sum_{ \substack{ \pi \in L^2_d(G, 1) \\ \pi_\infty \simeq J^\pm } } m(\pi) \dim( \pi_f^{K_f(\gn)} ).
\ees
The results of Section \ref{packets} allow us to rewrite this as

\bes
\beta^1_{(2)}( Y(\gn)) = \sum_{ \substack{ \xi \in L^2_d(H, \mu^{-1}) \\ \xi_\infty \in \Xi_\infty } } \sum_{ \substack{ \pi \in \Pi(\xi) \\ \pi_\infty \simeq J^\pm } } m(\pi) \dim( \pi_f^{K_f(\gn)} ).
\ees
Let $1_{K(\gn)} \in C( G(\A_f), 1)$ be the characteristic function of $Z(\A_f) K_f(\gn)$.  We have

\bes
\int_{G(\A_f) / Z(\A_f)} 1_{K(\gn)} d\overline{g} = c V(\gn)^{-1},
\ees
where $V(\gn)$ is as in (\ref{V}) and $c$ depends only on our choice of $K_v$ for $v \in S_f$, and so applying the upper bound $m(\pi) \le 1$ gives

\be
\label{beta}
\beta^1_{(2)}( Y(\gn)) \ll V(\gn) \sum_{  \substack{ \xi \in L^2_d(H, \mu^{-1}) \\ \xi_\infty \in \Xi_\infty } } \sum_{ \pi \in \Pi(\xi) } \tr( \pi_f( 1_{K(\gn)} ) ).
\ee

We now transfer $1_{K(\gn)}$ to a function $1^T_{K(\gn)} = \otimes_v 1^T_{K_v(\gn)} \in C( H(\A_f), \mu^{-1})$.  If $v \in S_f$, we let $1^H_{K_v(\gn)} \in C(H_v, \mu_v^{-1})$ be any transfer of $1_{K_v(\gn)}$, and set $1^T_{K_v(\gn)} = 1^H_{K_v(\gn)}$.  When $v \notin S$, we let $K^H_v$ be a hyperspecial maximal compact subgroup of $H_v$, and let $K^H_v(\p^n)$ be its standard principal congruence subgroups.  We define $1_{K^H_v(\gn)} \in C(H_v, \mu_v^{-1})$ to be the function supported on $Z_v K^H_v(\gn)$ and equal to 1 on $K^H_v(\gn)$, which is well defined as $\mu_v$ was assumed to be unramified, and set $1^T_{K_v(\gn)} = Nv^{-2 \ord_v \gn} 1_{K^H_v(\gn)}$.  When $v$ is split, the character identity

\be
\label{Kcharsplit}
\tr( \pi^n(\xi_v)( 1_{K_v(\gn)} ) ) = \tr( \xi_v(1^T_{K_v(\gn)}) )
\ee
may be directly verified.  When $v$ is inert, the character identity

\be
\label{Kcharinert}
\tr( \pi^n(\xi_v)( 1_{K_v(\gn)} ) ) + \tr( \pi^s(\xi_v)( 1_{K_v(\gn)} ) ) = \tr( \xi_v(1^T_{K_v(\gn)}))
\ee
follows from (\ref{charinert}) and the following proposition.

\begin{prop}

If $v \notin S$ is inert, the functions $1_{K_v(\gn)}$ and $Nv^{-2 \textup{ord}_v \gn} 1_{K^H_v(\gn)}$ are a transfer pair.

\end{prop}

\begin{proof}

Note that our assumption that $v \notin S$ implies that $E/F$ and $\mu$ are unramified at $v$ and that $v \nmid 2$.  Consequently, the proposition is a slight generalisation of the fundamental lemma for the pair $(U(3), U(2) \times U(1))$ in the special case of the identity in the Hecke algebra proven by Blasius-Rogawski \cite{BR} and Kottwitz \cite{K}, Flicker \cite{F1}, and Mars (\cite{F}, Section 2.I.6.).  We have checked that it may be proven by minor modifications to the arguments of Blasius-Rogawski-Kottwitz or Mars; we describe them in the case of Mars' proof.

As we shall work only at the place $v$ throughout the proof, we drop the subscript $v$ and assume all objects to be local at $v$.  We let $\pi$ be a prime element of $\cO_F$, and $q$ the order of its residue field.  We must verify the relation

\be
\label{transfer}
\Delta_{G/H}(\gamma) \Phi^\kappa(\gamma, 1_{K(\pi^r)}) = q^{-2r} \Phi^\text{st}(\gamma, 1_{K^H(\pi^r)})
\ee
for all semisimple $\gamma \in H$ whose inclusion in $G^*$ is regular.  It may be easily seen that it suffices to verify (\ref{transfer}) if we replace $1_{K(\pi^r)}$ and $1_{K^H(\pi^r)}$ by the characteristic functions of $K(\pi^r)$ and $K^H(\pi^r)$, and we henceforth do so.  This implies that both sides of (\ref{transfer}) vanish unless the $U(1)$ component of $\gamma$ lies in $U(1,\pi^r)$, and so after multiplying $\gamma$ by an element of $Z$ we may assume without loss of generality that the $U(1)$ component of $\gamma$ is 1.

Given an element $\gamma'$ in the stable conjugacy class $\cO_\text{st}(\gamma)$ and a hyperspecial maximal compact $K \subset G^*$, Mars computes the orbital integrals

\bes
\int_{G^* / G^*_{\gamma'} } 1_K( g \gamma' g^{-1} ) dg
\ees
comprising $\Phi^\kappa(\gamma, 1_{K})$ by counting self-dual lattices $\Lambda \subset E^3$ fixed by $\gamma'$.  The case in which $H_\gamma$ is isotropic is trivial, and in the remaining cases the centraliser $Y$ of $\gamma$ in $M_3(E)$ is isomorphic to $E^3$ or $E \times EL$, where $L/F$ a ramified quadratic extension and $EL$ is the compositum of the two fields.
\\

If $Y = E \times EL$, Mars lets $t \in EL^\times$ be one of the eigenvalues of $\gamma$ distinct from 1.  He writes $t = t_1 + t_2 w$ with $w = \sqrt{\pi} \in L$ and $t_i \in E$, and defines $A = \ord_E(t_1 - 1)$ and $B = \ord_E(t_2)$.  We then have

\be
\label{tfactor1}
\Delta_{G/H}(\gamma) = (-q)^{-n} \quad \text{where} \quad n = \min(2A, 2B+1),
\ee
see \cite{K}, Section 3.  Mars then counts lattices $\Lambda$ that satisfy $\Lambda^* = \Lambda$ and $\gamma \Lambda = \Lambda$ in terms of $A$ and $B$.  To calculate orbital integrals of the characteristic function of $K(\pi^r)$, we must instead count lattices that satisfy $\Lambda^* = \Lambda$, and such that $\gamma \Lambda = \Lambda$ and $\gamma$ acts trivially on $\Lambda / \pi^r \Lambda$; this latter condition is equivalent to $(\gamma - 1)\Lambda \subseteq \pi^r \Lambda$.

We now describe the key changes in Mars' formulas that follow from changing $(\gamma - 1)\Lambda \subseteq \Lambda$ to $(\gamma - 1)\Lambda \subseteq \pi^r \Lambda$.  Condition (2) on p. 298 now reads

\bes
t-1 \in \pi^r \cO_{EL}(n) \quad \text{and} \quad t-1 \in c \pi^r \cO_{EL}(n).
\ees
Note that the first condition arose because $(t-1)N_2 \subseteq \pi^r M_2$ implies $(t-1)N_2 \subseteq \pi^r N_2$.  In particular, this implies that $n \le B-r$.  Condition (*) on p. 299 now reads

\bes
\begin{array}{lll} \xi \eta^{-1} \pi^{m-n} t_2 \equiv t_1 - 1 & \text{mod } \pi^{2m+r} \cO_E & \text{in case 1}, \\
 & \text{mod } \pi^{2n + 1 +r} \cO_E & \text{in case 2}.
\end{array}
\ees
Following through Mars' computations, one sees that changing the conditions on $\Lambda$ in this way have the effect of replacing $A$ and $B$ with $A-r$ and $B-r$ in all his formulae for the number of $\Lambda$, and so we have

\bes
\Phi^\kappa(\gamma, 1_{K(\pi^r)}) = (-q)^{n-2r} \frac{ q^{B+1-r} -1}{ q-1}.
\ees
A similar computation in the case of $U(2)$ gives

\bes
\Phi^\text{st}(\gamma, 1_{K^H(\pi^r)}) = \frac{ q^{B+1-r} -1}{ q-1},
\ees
and combining these with (\ref{tfactor1}) gives (\ref{transfer}).
\\

When $Y = E^3$, Mars denotes the eigenvalues of $\gamma$ by $1$, $t_2$, $t_3 \in E^1$.  He defines $A = \ord_E(t_2 - t_3)$, $B = \ord_E(1 - t_3)$, and $C = \ord_E(1 - t_2)$, so that

\bes
\Delta_{G/H}(\gamma) = (-q)^{-B-C}
\ees
(see \cite{BR}, Section 6.3).  Condition (b) on p. 302 changes to $t_{23} M_{23} = M_{23}$ and $(t_{23}-1) N_{23} \subseteq \pi^r M_{23}$, and the equations in braces at the top of p. 303 now read

\begin{align*}
& n_2 + n_3 \ge n_1, \quad n_1 + n_3 \ge n_2, \quad n_1 + n_2 \ge n_3 \\
& \tfrac{1}{2}(n_2 + n_3 - n_1) \le A-r, \quad \tfrac{1}{2}(n_1 + n_3 - n_2) \le B-r, \\
& \tfrac{1}{2}(n_1 + n_2 - n_3) \le C-r, \\
& N_{E/F}(u) \in \nu_2 \nu_3^{-1} \pi^{n_2 - n_3 - \ord_E(\nu_2) + \ord_E(\nu_3)} + \pi^{n_2 - n_1} \cO_E, \\
& \quad (+ \pi^{n_2-n_1} \cO_E^\times \text{ if } n_1 + n_2 > n_3 \text{ and } n_1 + n_3 > n_2), \\
& (t_2 - 1) N_{E/F}(u) + (t_3-1) \nu_2 \nu_3^{-1} \pi^{n_2 - n_3 - \ord_E(\nu_2) + \ord_E(\nu_3)} \in \pi^{n_2 + r} \cO_E.
\end{align*}
As before, this has the effect of replacing $A$, $B$ and $C$ with $A-r$, $B-r$ and $C-r$ in all the lattice-counting formulae, and we obtain

\bes
\Phi^\kappa(\gamma, 1_{K(\pi^r)}) = (-q)^{B+C-2r} \frac{ q^{A-r}(q+1) - 2}{ q-1}.
\ees
Likewise, we have

\bes
\Phi^\text{st}(\gamma, 1_{K^H(\pi^r)}) = \frac{ q^{A-r}(q+1) - 2}{ q-1},
\ees
which completes the proof of the proposition.

\end{proof}

The identities (\ref{Kcharsplit}) and (\ref{Kcharinert}) and our description of the packet $\Pi(\xi)$ imply that

\bes
\sum_{ \pi \in \Pi(\xi) } \tr( \pi_f( 1_{K(\gn)} ) ) = 2\tr( \xi_f( 1^T_{K(\gn)}) ),
\ees
so that (\ref{beta}) becomes

\be
\label{beta1}
\beta^1_{(2)}( Y(\gn)) \ll V(\gn) \sum_{  \substack{ \xi \in L^2_d(H, \mu^{-1}) \\ \xi_\infty \in \Xi_\infty } } \tr( \xi_f( 1^T_{K(\gn)}) ).
\ee

Any $\xi \in L^2_d(H, \mu^{-1})$ is of the form $\xi_\theta = (\theta \circ \det_0) \otimes ( \theta^{-2} \mu^{-1} \circ \lambda)$ for some character $\theta \in \widehat{I}_E^1$, and the condition that $(\xi_\theta)_\infty \in \Xi_\infty$ restricts $\theta_\infty$ to a finite set $\Theta_\infty$.  We define the conductor $\gf_\theta$ of $\theta$ to be the largest ideal $\gm$ such that $\theta$ is trivial on $U(1, \gm)$.

Assume that $\theta \in \widehat{I}_E^1$ satisfies $\theta_\infty \in \Theta_\infty$ and $(\xi_\theta)_f( 1^T_{K(\gn)}) \neq 0$.  For $v \in S_f$, the condition $(\xi_\theta)_v( 1^T_{K_v(\gn)}) \neq 0$ and the fact that $1^T_{K_v(\gn)}$ is a smooth function that is independent of $\gn$ imply that $\ord_v \gf_\theta$ is bounded by a constant depending only on $K_v$.  If $v \notin S$, it may be easily seen that $(\xi_\theta)_v( 1^T_{K_v(\gn)}) \neq 0$ if and only if $\ord_v \gf_\theta \le \ord_v \gn$.  Consequently, there exists an ideal $\ga \subseteq \cO$ that is divisible only by primes in $S_f$ such that $\gf_\theta | \ga \gn$.  The number of characters with $\theta_\infty \in \Theta_\infty$ and $\gf_\theta | \ga \gn$ is $\sim | U(1, \cO) : U(1, \gn)|$, and for each $\theta$ we have

\bes
\tr( (\xi_\theta)_f( 1^T_{K(\gn)}) ) \ll N\gn^{-2} | U(2, \cO) : U(2, \gn)|^{-1}.
\ees
Combining these bounds with (\ref{beta1}) and subtituting the definition of $V(\gn)$, we obtain

\bas
\beta^1_{(2)}( Y(\gn)) & \ll \frac{ | U(1, \cO) : U(1, \gn)| | U(3, \cO) Z(\A_f) : U(3, \gn) Z(\A_f)| }{ N\gn^2 | U(2, \cO) : U(2, \gn)| } \\
& = \frac{ | U(3, \cO) : U(3, \gn)| }{ N\gn^2 | U(2, \cO) : U(2, \gn)| }.
\end{align*}
%
The formulas for the order of the groups $GL_3$ and $U(3)$ over a finite field \cite{A} imply that this is $\ll N\gn^3$, which completes the proof.

\subsection{The lower bound}
\label{prooflower}

Let $\xi^0_\infty \in \Xi_\infty$ be the character that is equal to $(\det_0)^{-t_v-1} \lambda^1$ at every infinite place $v$, so that $\Pi_{v_0}(\xi^0_{v_0}) = \{ J^+, D^- \}$.  Define

\bes
\Theta(\gn) = \{ \theta \in \widehat{I}_E | \gf_\theta = \gn, (\xi_\theta)_\infty = \xi_\infty^0 \} \quad \text{and} \quad \Xi(\gn) = \{ \xi_\theta | \theta \in \Theta(\gn) \}.
\ees
As $\gn$ was assumed relatively prime to $S_f$, $\theta \in \Theta(\gn)$ is unramified at $S_f$ and hence trivial at all nonsplit $v \in S_f$.  Because $E/F$ is CM, the elements $x \in \cO_E$ with $Nx = 1$ are exactly the roots of unity in $E$, and it follows that $|\Xi(\gn)| = | \Theta(\gn) | \gg N\gn$.

For nonsplit $v \in S_f$, choose $K_v$ so that $\pi^n(1_v)^{K_v}$ and $\pi^s(1_v)^{K_v}$ are both nonzero.  For split $v \in S_f$ and $\xi \in \Xi(\gn)$, $\pi^n(\xi_v)$ is the principal series representation $I( \xi_v \otimes \det_0 \circ \mu_w)$.  We see that we may choose $K_v$ so that $\pi^n(\xi_v)^{K_v} \neq 0$ for all unramified $\xi_v$.  Matsushima's formula and the results of Section \ref{packets} once again imply that

\bes
\beta_{(2)}^1(Y(\gn)) \ge \sum_{\xi \in \Xi(\gn)} \sum_{ \substack{ \pi \in \Pi(\xi) \\ \pi_\infty = J^+ } } m(\pi) \dim( \pi_f^{K_f(\gn)} ).
\ees
Let $\xi \in \Xi(\gn)$, and let $I$ be a finite set of inert places disjoint from $S$.  Then, because we assumed there was at least one nonsplit $v \in S_f$, there exists $\pi_I \in \Pi(\xi)$ with $\pi_\infty = J^+$ and $m(\pi) = 1$, and such that the set of $v \notin S$ with $\pi_v = \pi^s(\xi_v)$ is exactly $I$.  We have assumed that $\pi_v^{K_v} \neq 0$ for all $v \in S_f$, and so $\pi_I$ makes a contribution of at least

\bes
\prod_{v \in I} \dim( \pi^s(\xi_v)^{K_v(\gn)} ) \prod_{v \notin S \cup I} \dim( \pi^n(\xi_v)^{K_v(\gn)} )
\ees
to $\beta_{(2)}^1(Y(\gn))$.  Summing over $I$, we obtain

\bes
\beta_{(2)}^1(Y(\gn)) \ge \prod_{ \substack{ v \notin S \\ v \text{ split} } } \dim( \pi^n(\xi_v)^{K_v(\gn)} ) \prod_{ \substack{ v \notin S \\ v \text{ inert} } } ( \dim( \pi^n(\xi_v)^{K_v(\gn)} ) + \dim( \pi^s(\xi_v)^{K_v(\gn)} ) ).
\ees
We now define $1^S_{K(\gn)} \in C(G(\A^S), 1)$ to be the characteristic function of $\otimes_{v \notin S} K_v(\gn) Z(F_v)$, and let $1^{S,T}_{K(\gn)} \in C(H(\A^S), \mu^{-1})$ be the product over the places $v \notin S$ of the transfers defined in Section \ref{proofupper}.  Applying the character identities (\ref{Kcharsplit}) and (\ref{Kcharinert}) and summing over $\Xi(\gn)$ gives

\bes
\beta_{(2)}^1(Y(\gn)) \gg V(\gn) \sum_{\xi \in \Xi(\gn)} \tr( \xi^S( 1^{S,T}_{K(\gn)} ) ).
\ees
We have

\bes
\tr( \xi^S( 1^{S,T}_{K(\gn)} ) ) \gg N\gn^{-2} | U(2, \cO) : U(2, \gn)|^{-1}
\ees
when $\xi \in \Xi(\gn)$, and reasoning as in the case of the upper bound gives $\beta_{(2)}^1(Y(\gn)) \gg N\gn^3$.

\end{document}